\documentclass{amsart}

\usepackage{amssymb}

\usepackage{url} 

\usepackage{times}

\newtheorem{thm}{Theorem}[section]

\newtheorem{prop}[thm]{Proposition}
\newtheorem{lem}[thm]{Lemma}

\theoremstyle{remark}
\newtheorem{rem}[thm]{Remark}

\theoremstyle{definition}
\newtheorem{defi}[thm]{Definition}

\numberwithin{equation}{section}
\numberwithin{thm}{section}
\newcommand{\FF}{\mathcal{F}}
\newcommand{\R}{\mathbb{R}}

\newcommand{\lcp}{ \mathop{\mathrm{l.c.p.}}}
\newcommand{\sm}{{\,\cup\,}}

\def\trans{{\mathbin{\cap{\mkern-8mu}\mid}\,}}
\newcommand{\mycat}{\mathop{\mathrm{cat}}}
\newcommand{\cat}{\mycat_\trans (M,\FF)}
\newcommand{\scat}{\mycat^s_\trans (M,\FF)}

\newcommand{\codim}{\mathop{ \mathrm{codim} }}
\newcommand{\id}{\mathrm{id}}

\begin{document}


\title[A cohomological lower bound]{A cohomological lower bound
for the transverse 
LS~category of a foliated manifold}

\author{E. Mac\'{\i}as-Virg\'os}
\address{
Departamento de Xeometria e Topoloxia. 
Facultade de Matem\'aticas\\
Universidade de Santiago de Compostela\\
15782-Spain.
}

\email{quique.macias@usc.es}

\thanks{Partially supported by FEDER and   Research Project MTM2008-05861 MICINN Spain.}

% Abstract 

\begin{abstract} Let $\mathcal{F}$ be a compact Hausdorff foliation on a compact manifold. \\ Let ${E_2^{>0,\bullet}}=\oplus\{E_2^{p,q}\colon p>0,q\geq 0\}$ be the subalgebra of cohomology classes with positive transverse
degree in  the $E_2$ term of the spectral sequence of the foliation. We prove that the saturated transverse Lusternik-Schnirelmann category of $\mathcal{F}$ is bounded below by the length of the cup product in    ${E_2^{>0,\bullet}}$. Other cohomological bounds are discussed.

\end{abstract}

\maketitle

% end Topmatter
 
% body of paper
\section{Introduction}
The {\em transverse} Lusternik--Schnirelmann category $\cat$ of a foliated
manifold $(M,\FF)$
was introduced in H. Colman's thesis 
\cite{COLMAN_Thesis, COLMAN_MACIAS_Topology}. This  notion (and the analogous one of saturated transverse category) has been studied by several authors in the last years \cite{COLMAN_Taia, 
COLMAN_HURDER_Transactions, HURDER_Taia, 
HURDER_WALCZAK, LANGEVIN_WALCZAK, WOLAK}.

In \cite{COLMAN_Thesis,COLMAN_MACIAS_Topology}
a lower bound for $\cat$ was given (see Theorem \ref{BASIC} below). It is related to the length of the cup
product in the basic cohomology
of the foliation and generalizes the
corresponding classical result for the
LS category of a manifold \cite{LENGTH_James}.

We shall prove in Theorem \ref{DERHAM} that another lower bound for the transverse category  is the length of the cup product in the De Rham
cohomology of the ambient manifold for degrees greater than the dimension $d = \dim \FF$ of
the foliation, i.e. $\lcp H^{>d}(M) \leq \cat$.

For compact Hausdorff foliations we
introduce a new lower bound for the {\it saturated} transverse LS category
$\scat$,    in terms of the spectral sequence of the
foliation (Theorem \ref{COMPACT}). 
Explicitly, let $E_r^{p,q}$, $0\leq p\leq \codim \FF$, $0 \leq q \leq \dim
\FF$, be the de Rham spectral sequence of
the foliated manifold $(M,\FF)$,  and let ${E_2^{>0,\bullet}}= \oplus_{p>0, q\geq 0}E_2^{p,q}$ be the subalgebra of
cohomology classes with positive
transverse degree in the term $E_2$. Then the saturated transverse LS category of $(M,\FF)$ is bounded
below by the length of the cup product in ${E_2^{>0,\bullet}}$.

The interest of this result is that the $E_2$ terms of the spectral sequence
of a riemannian foliation on a compact manifold are known to be
finite dimensional \cite{SUSO_Illinois, HECTOR_KACIMI, SARKARIA, SERGIESCU_Spectral}.

During the preparation of this manuscript, S. Hurder informed the author that
he and H. Colman had found independently  analogous
results  for
the {\it tangential} LS category of a foliation (for the $E_1$ term) \cite{COLMAN_HURDER_AMS}, thus improving known
results from H.~Colman and the author
\cite{COLMAN_MACIAS_London} and from W.~Singhof and E.~Vogt \cite{SINGHOF_VOGT}.
As a matter of fact, the corresponding lower bound for the tangential category
of any foliation
follows easily from our constructions (Theorem \ref{TANGENTIAL}).\\

The author acknowledges several fruitful comments from J.~A.~\'Alvarez-L\'opez
and D. Tanr\'e.

\section{Transverse LS category}
Let $(M,\FF)$ be a $\mathcal{C}^\infty$ foliated manifold. An open subset $U\subset
M$ is said to be {\em transversely categorical}
if the inclusion factors through some leaf $L$ up to a foliated homotopy. That
is,
when we consider
the induced foliation $\FF_U$ on $U$, there exists  a leaf $L$ and a $\mathcal{C}^\infty$ homotopy
$H \colon U \times \R \to M$  such that:
each $H_t \colon U \to M$ sends leaves into leaves;
$H_0$ is the inclusion $U\subset M$; and $H_1( U) \subset L$.

\begin{defi}
The {\it transverse LS category} $\cat$ of the foliation is the least
integer $k\geq 0$ such that $M$ can be covered by $k+1$
transversely categorical open sets. 
\end{defi}
If such a covering does not exist we
put $\cat = \infty$.
Notice that since adapted charts are categorical, we have $\cat < \infty$
when the manifold $M$ is compact.

\begin{rem} In the original paper \cite{COLMAN_MACIAS_Topology}, the definition above
corresponds
to $\mycat +1$, but presently we follow the more extended convention that
contractible
spaces have null LS category. Coherently, we take the
length of the cup product (shortly $\lcp$) of an algebra to be the maximum
$k\geq 0$ for which
there exists some product
$a_1  \cdots  a_k \neq 0$ ($k=0$ means that all products are null).
\end{rem}

\section{Basic cohomology}

The following cohomological lower bound for $\cat$ involving the basic
cohomology of
the foliation was given in \cite{COLMAN_MACIAS_Topology}.

Recall that the {\it basic} cohomology $H_b=H(M,\FF)$ is
defined by means of the complex
$\Omega_b=\Omega(M,\FF)$
of {\em basic} forms, that is differential  forms $\omega\in \Omega(M)$
such  that
$i_X\omega =0 =i_X d\omega$ for any vector field $X$ tangent to the
foliation. Then the inclusion $ \Omega_b\subset
\Omega(M)$ induces a morphism $\pi^*\colon H_b \to H(M)$, which in general
is  not injective. Let
$\pi^*H_b^{>0}$ be the image of the basic cohomology in positive degrees.

\begin{thm}[\cite{COLMAN_MACIAS_Topology}]\label{BASIC}
For any foliated manifold, the
transverse LS category is
bounded below by the length of the cup product in $\pi^*H_b^{>0}$.
\end{thm}

\begin{proof}
The argument is standard, but we include it for later use.
Let $U\subset
M$ be a transversely categorical open set.
Since two  homotopic maps (by a foliated homotopy) induce the same morphism in
basic cohomology, the induced map
$H^{>0}_b(M)
\to H_b^{>0}(U)$ is null because
$H^{>0}_b(L) =0$ for any leaf $L$. Hence the map
$H_b^{>0}(M,U)
\to H_b^{>0}(M)$ in the long exact sequence (for basic cohomology) of the
pair $(M,U)$ is surjective. Now, let $k= \cat$ and let
$U_0,\dots,U_k$ be
$k+1$ transversely categorical open subsets covering
$M$. If
$\omega_0,\dots,\omega_k$ are basic cohomology classes of positive degrees,
then each $\omega_i$ can be lifted to some
$\xi_i \in H_b(M,U_i)$, and $$\pi^*\xi_0 \sm \cdots \sm \pi^*\xi_k \in
H(M,U_0 \sm \cdots \sm U_k) = H(M,M) = 0,$$
hence $\pi^*(\omega_0 \sm \cdots \sm \omega_k) =0$.
\end{proof}

It should be noted that in general it is not possible to define a cup product in the
relative basic
cohomology, due to the lack
of adequate partitions of unity. So we cannot use the
length of the basic cohomology as a  (better) lower bound  for the
transverse category.
However, this can be done for a
particular class of foliations (compact Hausdorff foliations), and the {\it
saturated}
transverse LS category, as we shall prove in Theorem
\ref{COMPACT}.

\section{New lower bounds}
Another lower bound for $\cat$ is almost immediate from the definitions.
It was suggested by the analogous result
from W. Singhof and E. Vogt for the {\it tangential} category of a foliation
\cite{SINGHOF_VOGT} and was the motivational idea
for the present paper.

Let $H^{>d}(M)$ be the de Rham cohomology of the ambient manifold in
degrees greater than the dimension of the foliation. Clearly,
for a transversely categorical open set $U$ the map $H^{>d}(M) \to
H^{>d}(U)$ induced by the inclusion is the zero map,
because it factors through $H^{>d}(L)=0$, $d=\dim L$. Then the standard
argument (this time for de Rham cohomology of $M$) applies, and we have
proved:

\begin{thm}\label{DERHAM} The transverse LS category of a foliation is
bounded below by the length of the cup product in
$H^{>d}(M)$, for $d = \dim \FF$.
\end{thm}

When
comparing the latter  result with  Theorem \ref{BASIC}, one realizes that both
 bounds
can be explained in terms of the (de Rham) spectral sequence $E_r^{p,q}
\Rightarrow H(M)$ of the
foliation.

\section{The spectral sequence of a foliated manifold}\label{SPECTRAL}
This is a very well known algebraic tool
which  has been extensively studied; we refer the reader for instance to
 \cite{MACIAS, MASA, SARKARIA, VAISMAN}  among many others. When the foliation is defined by a
fibre bundle, one obtains the Serre spectral sequence, written for de Rham cohomology as in  \cite{HATTORI}. 

\subsection{Basic notions}
Let us recall some basic notions. Let $\Omega(M)$ be the de Rham complex of
the ambient manifold
$M$. We define a decreasing filtration $F^p\Omega^r(M)$, $0 \leq p \leq r$,  of
$\Omega^r(M)$,
$0 \leq r \leq \dim M$,   by the condition:
$\omega \in F^p\Omega^r$ if and only if $i_{X_0}\cdots i_{X_{r-p}}\omega = 0$
whenever the vector fields $X_0,\dots, X_{r-p}$ are  tangent to the foliation.

Put $E_0^{p,q} =
F^{p}\Omega^{p+q}/F^{p+1}\Omega^{p+q}$. Then the exterior differential $d$
induces, for each $0 \leq p \leq \codim \FF$, a differential $d_0^{p,q}
\colon E_0^{p,q} \to E_0^{p,q+1}$, $0 \leq q \leq \dim \FF$, whose
cohomology groups are denoted by $E_1^{p,q}$.

Often we denote $\Omega^{p,q} =E_0^{p,q}$. By taking any
distribution ${\mathcal N}$ complementary to the foliation
we have $TM = T\FF \oplus {\mathcal N}$, so $\Omega(M) = \Gamma \Lambda (T^*M)$
  is isomorphic to
$\Gamma \Lambda ({\mathcal N^* })\otimes \Gamma \Lambda (T^*\FF)$, hence
$\Omega^{p,q} \cong
\Gamma \Lambda^p({\mathcal N^*}) \otimes \Gamma \Lambda^q (T^*\FF)$.
We shall use the well known fact that if $\omega \in \Omega^{p,q}$ then $d
\omega \in
\Omega^{p+q+1}(M) $ decomposes (because $d^2=0$) in three parts
$[d\omega]^{p+1,q} + [d\omega]^{p,q+1} +[d\omega]^{p+2,q-1}$.

Now, the exterior differential induces morphisms
$d_1 \colon E_1^{p,q} \to E_1^{p+1,q}$ such that $(d_1)^2 = 0$, and we denote
by $E_2^{p,q} = H^p(E_1^{\bullet,q},d_1)$ its cohomology groups. The group
$H_b^{p}$ of basic
cohomology  corresponds to the term $E_2^{p,0}$.

 These are the first steps
in order to define  morphisms $d_r^{p,q} \colon E_r^{p,q}
\to E_r^{p+r, q-r+1}$, $r\geq 0$, and a spectral sequence $E_{r+1} = H(E_r,d_r)$
which converges
to the de Rham cohomology $H(M)$ of the ambient manifold in a finite number
of steps. 
%Then $H(M)$ is isomorphic to $\oplus_{}E_\infty^{p,q}$, $0\leq p
%\leq\codim \FF$,
%$0\leq q \leq \dim \FF$.

We shall also need the well known fact that the exterior product of differential
forms
induces a well defined cup product in each term $E_r$ of the spectral sequence,
which is
compatible with the bigraduation. Moreover, the morphisms $d_r$ satisfy the
usual
derivation formula
$d_r(\alpha \sm \beta) = d_r\alpha \sm \beta + (-1)^\alpha \alpha \sm d_r\beta$.
More details about this cup product can be found for instance in the books of A.~T. ~Fomenko {\it et
al.} \cite{FOMENKO} or  P.~Selick \cite{SELICK}.

\subsection{Foliated homotopy}

The following result appeared in J.~A. \'Alvarez L\'opez's thesis \cite{SUSO}.

\begin{lem} \label{HOMOTOPY}
Let $f \colon (M,\FF) \to (M^\prime,\FF^\prime)$ be a
foliated (i.e. sending leaves into leaves)
$\mathcal{C}^\infty$ map between foliated manifolds. Then:
\begin{enumerate}
\item
The induced cohomology morphism $f^* \colon H(M^\prime) \to H(M)$
preserves the filtration and hence it maps $E_r^{p,q}(M^\prime)$ into
$E_r^{p,q}(M)$, $0\leq r $, for all $p,q$;
\item
If $g$ is another map which is homotopic to $f$   by a
foliated
homotopy   then $f^* = g^*$ on each $E_2^{p,q}$.
\end{enumerate}
\end{lem}

\begin{proof} Part (1) is clear because  $X \in \Gamma T\FF$ implies $f_*(X) \in \Gamma T\FF^\prime$.

For part (2) we can suppose $M^\prime = M\times \R$, endowed with the foliation
$L\times \{t\}$, $L\in \FF$, and $f=i_0$, $g=i_1$, where $i_t(x)=(x,t)$. The maps $i_0,i_1$ are homotopic  
by the foliated homotopy $i_t$. Hence the morphisms $i_0^*, i_1^* \colon
\Omega^r(M^\prime) \to
\Omega^r(M)$ are algebraically homotopic by the application
$H \colon \Omega^{r+1}(M^\prime) \to \Omega^{r}(M)$ given by
$H\omega = \int_0^1 i_{\partial_t}\omega dt$. It  induces an algebraic homotopy
$E_1^{p+1,q}(M^\prime) \to E_1^{p,q}(M)$, $p+q=r$, between
$i_0^*, i_1^*$ at the level $E_1$, hence $i_0^*=i_1^*$ at the level $E_2$.
\end{proof}

\section{Compact Hausdorff foliations}
A foliation $\FF$ on the compact manifold $M$ is said to be {\it compact Hausdorff}
if every leaf is compact and the space of leaves is Hausdorff  \cite{EPSTEIN}. Several
interesting results and computations
of the saturated transverse LS category $\scat$ in this setting have been obtained
by H.~Colman, 
S.~Hurder and  P.~G. Walczak \cite{COLMAN_HURDER_Transactions, HURDER_WALCZAK}. Recall that $\scat$ is defined \cite{COLMAN_MACIAS_Topology}   by considering
transversely categorical  open sets which are {\it saturated} (i.e.  a
union of leaves).

We have the following result:

\begin{thm}\label{COMPACT} Let $\FF$ be a compact Hausdorff foliation on the
manifold $M$. Let $E_2^{>0,\bullet} = \oplus_{p>0, q\geq 0}E_2^{p,q}$ be the subalgebra of classes in the $E_2$ term of the spectral sequence with positive transverse degree. Then
$\lcp E_2^{>0,\bullet}\leq \scat$.
\end{thm}

\begin{proof}  
We present here the guidelines of the proof. The details
are rather technical (although
not sophisticated), so we have moved them to Section
\ref{RELATIVE}.

First, since there are partitions of unity which are constant along the leaves,  we
have a Mayer-Vietoris sequence
for saturated open sets, which is excisive in the $E_1$ level.  Hence it is possible
to define a cup
product in the relative $E_2$ terms of the spectral sequence. 
Finally, part (2) of Lemma \ref{HOMOTOPY}
allows us to apply the standard cohomology argument cited in Theorem \ref{BASIC}
because
$E_2^{p,q}(L)=0$ for $p>0$ and any leaf $L$. \end{proof}

\begin{rem}[Fiber bundles] Our computation    has an application to the
classical LS category of a manifold. 
Let  $F\to E \to B$ be a smooth locally trivial fiber bundle
with connected fibers,
$E_r^{p,q}$ the corresponding spectral sequence \cite{HATTORI}, hence $E_2^{p,q}=H^p(B;\mathcal{H}^q(F))$.
The (classical)
Lusternik-Schnirelmann
category $\mycat\,B$ of the base equals $\scat$ for the foliation $\FF$ in $E$ defined by the fibers
\cite{COLMAN_MACIAS_Topology}, hence it is bounded below by the length of the cup
product in
$E_2^{p>0,\bullet} = \oplus_{p>0, q\geq 0}E_2^{p,q}$.
(However, since we are working with real coefficients, this bound is possibly
not better than the usual lower bound  $\lcp  H^{>0}(B) \leq \mycat\, B$). 
 \end{rem}

\section{Relative cup product} \label{RELATIVE}
In this section we develop the technical details in the proof of
Theorem \ref{COMPACT}. Probably many of them are folk, as we follow the ideas of the book of Bott-Tu \cite{BOTT_TU}. Since the proof is rather lengthly,
I have tried to write it in such a way that it becomes accessible to non specialists.

The crucial
point is the Mayer-Vietoris argument for the $E_1$ term in Proposition \ref{MAYER_VIETORIS}.
In fact, for $E_0$ it is known (see   Munkres's book \cite{MUNKRES})
that  a cup product
in the relative cohomology groups of a topological space can be defined for any excisive pair.

\subsection{Preliminaries}
It is an exercise to prove that for a compact Hausdorff foliation there
exist partitions of
unity which are constant along the leaves, for any 
finite covering by saturated open sets.

Let $W$ be a saturated open set.
Since the inclusion $W\subset M$ is a foliated map, we have induced
morphisms $(i_W)_r^*\colon (E_r(M),d_r) \to (E_r(W),d_r)$ between the
spectral sequences. Often we denote
$(i_W)_r^*\omega$ simply by $i^*\omega$ or $\omega_W$ when  there is no
risk of confusion.

\subsection{Relative cohomology}\label{REL_COHOM}
 Let  us define
$$E_1^{p,q}(M,W) = E_1^{p,q}(M)\oplus E_1^{p-1,q}(W)$$
endowed with the differential $\delta =\delta_1 \colon E_1^{p,q}(M,W) \to
E_1^{p+1,q}(M,W)$ given by
$$\delta (\mu, \omega) = (d_1\omega, i_1^*\mu - d_1\omega).$$

Since $\delta^2 =0$, we can define $E_2^{p,q}(M,W) \colon =
H^p(E_1^{\bullet,q}(M,W),\delta)$, and we have a long
exact sequence in cohomology,
$$\cdots \to E_2^{p-1,q}(W) \to E_2^{p,q}(M,W) \to E_2^{p,q}(M)
\stackrel{}{\to} E_2^{p,q}(W) \to \cdots,$$
induced by the (obvious) short exact sequence of complexes
$$ 0 \to (E_1^{\bullet-1,q}(W), -d_1) \to  (E_1^{\bullet,q}(M,W),\delta) \to
(E_1^{\bullet,q}(M), d_1) \to 0.$$

It is clear that this construction could already be done at the $E_0$
level.

\subsection{Connecting morphism}
Let us denote $\Omega^{p,q} = E_0^{p,q}$.
Let $U,V\subset M$ be saturated open sets, and
$\{\varphi_U,\varphi_V\}$ a smooth partition of unity
on $U\cup V$ subordinated to the open covering
$\{U,V\}$. If the functions $\varphi_U,\varphi_V$ are constant along the leaves,
then
$d\varphi_U, d\varphi_V \in \Omega^{1,0}$.  The {\it connecting morphism}
$$\Delta \colon \Omega^{p,q}(U\cap V) \to \Omega^{p+1,q}(U\cup V)$$
 is defined by
$$\Delta(\omega) =
\left\{ \begin{matrix}
d\varphi_V \wedge \omega & {\mathrm on\ } U \cr
-d\varphi_U \wedge \omega & {\mathrm on\ } V \cr
\end{matrix}\right. .$$
This is well defined (we invite the reader to check it). Moreover it is a
morphism of complexes
$\Delta \colon (\Omega^{p,\bullet}, d_0) \to (\Omega^{p+1,\bullet}, -d_0) $
because for
$\omega \in \Omega^{p,q}$  we have
(for instance on $U$)
\begin{align*}
d_0\Delta(\omega) &=  \\
 {[d\Delta(\omega)]}^{p+1,q+1} = [d(d\varphi_V \wedge\omega)]^{p+1,q+1}
&=  \\
{[- d\varphi_V \wedge d \omega]}^{p+1,q+1} = -d\varphi_V \wedge [d\omega
]^{p,q+1} &=   \\
-d\varphi_V \wedge d_0\omega &=  - \Delta d_0\omega,
\end{align*}
and analogously on $V$. Then, since  $E_1^{p,q} = H^q(\Omega^{p,\bullet},d_0)$,
we have  induced
morphisms
\begin{equation}\label{CONNECT}
\Delta \colon E_1^{p,q}(U\cap V) \to E_1^{p+1,q}(U\cup V).
\end{equation}

\subsection{Mayer-Vietoris sequence}

Now we consider
the Mayer-Vietoris sequence
\begin{equation}\label{MV_E1}
0 \to E_1^{p,q}(U\cup V) \stackrel{i}{\to} E_1^{p,q}(U)\oplus E_1^{p,q}(V)
\stackrel{\pi}{\to}
E_1^{p,q}(U\cap V)
\to 0
\end{equation}
 defined in the usual way, that is 
$$i(\xi) =\left(
(i_U)_1^*(\xi),(i_V)_1^*)(\xi) \right)$$
 and
$$\pi(\alpha,\beta) = (i_{U\cap V})_1^*(\alpha) - (i_{U\cap V})_1^*(\beta).$$

Due to the existence of basic partitions of unity
we have:

\begin{lem}\label{MAYER_VIETORIS}
The Mayer-Vietoris sequence (\ref{MV_E1}) is exact.
\end{lem}

\begin{proof}
Let
\begin{equation}\label{MV_E0}
0 \to (\Omega^{p,\bullet}(U \cup V),d_0) \stackrel{i}{\to}
(\Omega^{p,\bullet}(U),d_0)\oplus (\Omega^{p,\bullet}(V),d_0)
\stackrel{\pi}{\to}
(\Omega^{p,\bullet}(U\cap V), d_0) \to 0
\end{equation}
be the usual Mayer-Vietoris short exact sequence for  the ambient manifold,
restricted to a fixed transverse degree $p$. The positive fact
is that the usual section of $\pi$ given by $S (\omega) = (\varphi_V
\omega, -\varphi_U \omega)$ is
in our setting a {\it morphism of complexes}, because our partitions of
unity are constant along the leaves,
i.e. $d\varphi_U, d\varphi_V \in \Omega^{1,0}$. In fact, for $\omega \in
\Omega^{p,q}(U\cap V)$ we have
\begin{align*}
dS(\omega) &= \\
 \left(d_0(\varphi_V \omega),-d_0(\varphi_U \omega)\right) 
&= \\
([d(\varphi_V \omega)]^{p,q+1},-[d(\varphi_U \omega)]^{p,q+1}) &=
&\\
([d\varphi_V \wedge \omega + \varphi_V \wedge
d\omega]^{p,q+1},-[d\varphi_U \wedge \omega
+ \varphi_U \wedge d\omega)]^{p,q+1}) &=  \\
(\varphi_V \wedge [d\omega]^{p,q+1}, -\varphi_U \wedge
[d\omega)]^{p,q+1}) &=  \\
S ([d\omega]^{p,q+1})  &=  S d_0 (\omega). 
\end{align*}
Then the long exact sequence associated to (\ref{MV_E0}) splits, and gives
the short exact sequence
(\ref{MV_E1}), because $E_1^{p,q} = H^q(\Omega^{p,\bullet},d_0)$. \end{proof}

Notice that the morphism $\Delta$ explicitly defined in (\ref{CONNECT})
induces in fact the connecting morphism of
the long exact sequence associated to (\ref{MV_E1}). Moreover the section $S$ at
the $E_0$ level
induces a section (also denoted $S$) such that $\pi S=\id$ in the $E_1$ sequence
(\ref{MV_E1}).
However,
$S$ is not yet a morphism of complexes, and the default
\begin{equation}\label{DEFAULT1}
\Delta = dS -Sd
\end{equation}
 is just the
connecting morphism defined above, as the reader can check.

On the other hand, the section $S$ on the right side of the sequence
(\ref{MV_E1})
induces another section $J = \id - S\pi$ for the {\it left side}, that is
$$J \colon E_1^{p,q}(U)\oplus E_1^{p,q}(V) \to E_1^{p,q}(U\cup V)$$ such that
$J i = \id$.
It is an exercise to prove that
\begin{equation}\label{DEFAULT2}
dJ -Jd = -\Delta \pi.
\end{equation}

\subsection{Cup product in relative cohomology}

Now we define a product (compatible with the absolute cup product)
\begin{equation}\label{CUP_E2}
\sm \colon E_2^{p,q}(M,U) \otimes E_2^{r,s}(M,V) \to E_2^{p+r,q+s}(M, U
\cup V)
\end{equation}
in the $E_2$ term of the spectral sequence. It
will be induced by a product in the relative $E_1$ level (defined in section
\ref{REL_COHOM}), that is
$$\sm \colon E_1^{p,q}(M,U) \otimes E_1^{r,s}(M,V) \to E_1^{p+r,q+s}(M, U
\cup V).$$ This latter  product is given by
\begin{equation}\label{CUP_E1}
(\mu,\alpha) \sm (\nu, \beta) = (\mu \sm \nu, \xi),
\end{equation}
where $\xi \in E_1^{p+r-1,q+s}(U \cup V)$
is explicitly written as
\begin{equation}\label{CUP}
\xi = J\left(\alpha \sm \nu_U, (-1)^\mu \mu_V \sm \beta\right) + (-1)^r \Delta
\left(\alpha_{U\cap V} \sm
\beta_{U\cap V} \right).
\end{equation}
Here, $$\Delta \colon E_1^{p-1+r-1,q+s}(U\cap V) \to
E_1^{p+r-1,q+s}(U\cup V)$$ is the connecting morphism defined in
(\ref{CONNECT}),
and $$J \colon E_1^{p+r-1,q+s}(U)\oplus E_1^{p+r-1,q+s}(V) \to
E_1^{p+r-1,q+s}(U\cup V)$$ is the
left section $J$ considered in equation \ref{DEFAULT2}.
When there is no risk of confusion we shall understand that the classes $\mu, \beta,
\nu, \beta$ are adequately restricted, so we should simply write
$$\xi = J\left( \alpha \sm \nu, (-1)^\mu \mu \sm \beta\right) + (-1)^r \Delta \left(\alpha \sm
\beta\right).$$

Next we prove:
\begin{prop} \label{CONMUTES}The product (\ref{CUP}) induces a well defined
product in the $E_2$
terms.
\end{prop}

\begin{proof} %{PROOF of Proposition \ref{CONMUTES}}
Although it is possible to check that by hand, we shall present a more
elaborate proof,
which has the advantage of being valid for any level $E_r$ where the key
ingredients
$\Delta$,
$J$ will be defined.

Since the map (\ref{CUP_E2}) has the form
$$ H^p(E_1^{\bullet,q}, d_1) \otimes H^r(E_1^{\bullet,s}, d_1) \to
H^{p+r}(E_1^{\bullet,q+s},d_1),$$
what we need is a map
$$H^{p+r}(E_1^{\bullet,q} \otimes E_1^{\bullet,s}) \to H^{p+r}(E_1^{\bullet,
q+s}),$$
so we must only check that the morphism
$$\sm \colon E_1^{\bullet,q}(M,U) \otimes E_1^{\bullet,s}(M,V) \to
E_1^{\bullet,q+s}(M, U
\cup V)$$
given in (\ref{CUP_E1}) commutes with the differentials, where as usual the
left complex
is endowed with the differential
$d(A\otimes B) = d_1A \otimes B + (-1)^A A \otimes d_1 B$.

Then we have
\begin{align*}
\delta \Big( (\mu,\alpha) \sm (\nu,\beta) \Big) &=  \\
\delta \Big(\left( \mu\sm \nu, J(\alpha \sm \nu, (-1)^\mu  \mu \sm \beta)
+(-1)^\mu  \Delta(\alpha \sm
\beta) \right) \Big) &=  \\
\Big( d_1(\mu \sm \nu), \mu \sm \nu - d_1 J(\alpha \sm \nu, (-1)^\mu \mu \sm
\beta)
- (-1)^\mu d_1\Delta (\alpha \sm \beta) \Big). &&
\end{align*}

Now by using that $J(\mu \sm \nu,\mu \sm \nu) = \mu \sm \nu$, that
$d_1J = Jd_1 - \Delta \pi$ and that $d_1\Delta = - \Delta d_1$, we have that the
first coordinate
is (we write $d=d_1$)
\begin{equation}\label{LEFT}
d\mu \sm \nu + (-1)^\mu \mu \sm d\nu
\end{equation}
while the second is
\begin{align*}
J(\mu\sm \nu, \mu \sm \nu)
- Jd(\alpha \sm \nu, (-1)^\mu \mu \sm \beta) + &&\\
\Delta\pi(\alpha \sm \nu, (-1)^\mu \mu \sm \beta)
+ (-1)^\mu \Delta d(\alpha \sm \beta)  &= \\
J\left( \mu \sm \nu - d(\alpha \sm \nu), \mu \sm \nu - (-1)^\mu d(\mu \sm
\beta)\right)
+ &&\\
 \Delta\pi(\alpha\sm \nu, (-1)^\mu \mu\sm \beta) + (-1)^\mu \Delta(d\alpha\sm
\beta + (-1)^\alpha\alpha
\sm d\beta),
\end{align*}
that is
\begin{align}\label{MONSTER}
J\Big( \mu \sm \nu - d\alpha \sm \nu -(-1)^\alpha\alpha \sm d\nu,
\mu \sm \nu - (-1)^\mu d\mu \sm \beta - \mu \sm d\beta \Big)
+&&\nonumber\\
 \Delta\pi(\alpha\sm \nu, (-1)^\mu \mu\sm \beta) + &&\\
(-1)^\mu \Delta(d\alpha\sm \beta + (-1)^\alpha \alpha \sm d\beta). \nonumber
\end{align}
On the other hand, $\cup d\left( ( \mu,\alpha)\otimes (\nu,\beta) \right)$
equals
\begin{align*}
d(\mu,\alpha) \sm (\nu,\beta) + (-1)^\mu (\mu,\alpha) \sm d(\nu,\beta)
&= \\
(d\mu, \mu -d\alpha) \sm (\nu, \beta) + (-1)^\mu (\mu,\alpha) \sm (d\nu, \nu -
d\beta)
&= \\
\Big(
d\mu \sm \nu, J\left(  (\mu -d\alpha) \sm \nu, (-1)^{d\mu }  d\mu \sm \beta
\right) +
(-1)^{d\mu } \Delta((\mu - d\alpha)\sm \beta)\Big) +&&\\
(-1)^\mu \Big( \mu \sm d\nu, J\left(\alpha \sm d\nu, (-1)^\mu \mu \sm
(\nu-d\beta)\right) +
(-1)^\mu \Delta(\alpha \sm (\nu -d\beta))\Big)
\end{align*}
whose first coordinate equals (\ref{LEFT}), while the second is
\begin{align*}
J\Big( \mu\sm \nu -d\alpha \sm \nu + (-1)^\mu \alpha\sm d\nu, (-1)^{d\mu }  d\mu
\sm \beta
+  \mu\sm \nu -  \mu \sm d\beta
 \Big) +&& \\
\Delta \Big( (-1)^{d\mu } \mu \sm \beta - (-1)^{d\mu }  d\alpha \sm \beta
+  \alpha\sm \nu -  \alpha\sm d\beta\Big),
\end{align*}
which equals (\ref{MONSTER}), because $(-1)^\mu = (-1)^{\alpha+1}$
and $(-1)^{d\mu} = (-1)^{\mu +1}$,
and we have done (!).\end{proof}

{\it Remark:\/} We have only  used the K\"unneth morphism 
$$H(E_U)\otimes H(E_V) \to
H(E_U\otimes E_V),$$
but in fact it is an isomorphism because the relative groups  $E_2(M,W)$ are
finite dimensional
for compact Hausdorff foliations on a compact manifold.

\section{Tangential LS category}
The following result is an easy consequence of our computation in Section
\ref{RELATIVE}.
It has been found independently by S. Hurder and H. Colman
\cite{COLMAN_HURDER_AMS}.

\begin{thm}\label{TANGENTIAL}
 Let $(M,\FF)$ be any foliated manifold. Then the tangential LS
category is bounded below by the length of the cup product in
$E_1^{\bullet,>0}=\oplus_{p\geq 0, q>0}E_1^{p,q}$, the
subalgebra of $E_1$ of cohomology classes with positive tangential degree.
\end{thm}

\begin{proof} 
Roughly speaking, the tangential LS category \cite{COLMAN_MACIAS_London} is defined by means of open sets
which deform to a
transversal {\it along the leaves}. As it   is well known \cite{SUSO}, this kind of {\it integrable}
homotopy
is an invariant of the $E_1$ term of the spectral sequence (compare with our
Lemma
\ref{HOMOTOPY}), hence the inclusion of a tangentially categorical open set
vanishes in
cohomology for positive tangential degrees.

But on the other hand the Mayer-Vietoris sequence is always exact in the
$E_0$ term,
so it is possible to define the relative $E_1$ cohomology, and the standard
argument applies.
\end{proof}

% bibliography

\end{document}